\documentclass[12pt,reqno]{amsart}
\usepackage{amsmath,amsfonts,amsthm,amssymb,color}
\usepackage[usenames,dvipsnames]{xcolor}
\usepackage[T1]{fontenc}
\usepackage{enumerate}
\usepackage{hyperref}
\hypersetup{
     colorlinks   = true,
     citecolor    = blue,
     linkcolor    = blue
}

\usepackage{pdfsync}{\tiny }

\usepackage[left=1in, right=1in, top=1.1in,bottom=1.1in]{geometry}
\numberwithin{equation}{section}

\newcommand{\E}{\mathbb{E}}
\newcommand{\R}{\mathbb{R}}
\newcommand{\nT}{\lfloor nT \rfloor}
\newcommand{\mT}{\lfloor mT \rfloor}
\newcommand{\nt}{\lfloor nt \rfloor}

\newcommand{\ns}{\lfloor ns \rfloor}
\newcommand{\mti}{\lfloor m t_i \rfloor}
\newcommand{\nti}{\lfloor nt_i \rfloor}

\newtheorem{thm}{Theorem}[section]
\newtheorem{lemma}[thm]{Lemma}

\newtheorem{prop}[thm]{Proposition}

\def\1{{\rm l}\hskip -0.21truecm 1}

\begin{document}

\title{Weak symmetric integrals with respect to the fractional Brownian motion}

 \date{\today}

\author{Giulia Binotto \and Ivan Nourdin \and David Nualart}

\address{Giulia Binotto: {\color{black}{Facultat de Matem\`atiques, Universitat de Barcelona, Gran Via de les Corts Catalanes 585, 08007 Barcelona, Spain.}} }
\email{gbinotto@ub.edu}
\thanks{G. Binotto was supported by the grant MTM2012-31192 from SEIDI, Ministerio de Economia y Competividad}

\address{Ivan Nourdin: \textcolor{black}{Rmath, Fstc, Luxembourg University, 6 rue Richard Coudenhove-Kalergi,
L-1359 Luxembourg}.}
\email{ivan.nourdin@uni.lu}
\thanks{I. Nourdin was partially supported  by the Grant F1R-MTH-PUL-15CONF (CONFLUENT) at Luxembourg University}

\address{David Nualart: Department of Mathematics, University of Kansas,   Lawrence, KS 66045, USA.}
\email{nualart@ku.edu}
\thanks{D. Nualart was supported by the NSF grant  DMS1512891 and the ARO grant FED0070445}


\date{\today}


\begin{abstract}
 The aim of this paper is to establish the weak convergence, in the topology of the Skorohod space, of the $\nu$-symmetric Riemann sums for functionals of the fractional Brownian motion \textcolor{black}{when the Hurst parameter takes the critical value $H=(4\ell+2)^{-1}$, where $\ell=\ell(\nu)\geq 1$ is the largest natural number satisfying $\int_0^1 \alpha^{2j}\nu(d\alpha)=(2j+1)^{-1}$ for all $j=0,\ldots,\ell-1$.} As a consequence, we derive a change-of-variable formula in distribution, where the correction term is a stochastic integral with respect to a Brownian motion that is independent of the fractional Brownian motion.
\end{abstract}

\maketitle

\section{Introduction}\label{Intro}
Suppose that $B^H= \{ B^H_t, t\ge 0\}$ is a  fractional Brownian motion (fBm) with Hurst parameter $H\in (0,1)$, that is, $B^H$ is a centered
Gaussian process with covariance given by
    \begin{equation}\label{cov}
      R(s,t) := \E[B^H_s B^H_t] = \frac12 (s^{2H}+t^{2H}-|t-s|^{2H}),
    \end{equation}
    for any $s,t \ge 0$. \textcolor{black}{When $H<\frac12$,} it is well-known that the integral $\int_0^t g(B^H_s) dB^H_s$  does not exist \textcolor{black}{in general} as a path-wise Riemann-Stieltjes integral. In the pioneering work \cite{GNRV}, Gradinaru, Nourdin, Russo and Vallois proved that this  integral can be defined  as the limit in probability of  suitable symmetric Riemann sums if the Hurst parameter is not too small. Let us briefly describe the main  contribution of \cite{GNRV}.

  Let $\nu$ be a symmetric  probability measure on $[0,1]$, \textcolor{black}{meaning that} $\nu(A) =\nu(1-A)$ for any Borel set $A\subset [0,1]$.  Given a  continuous function $g:\R \rightarrow \R$,  consider   the $\nu$-symmetric  Riemann sums of $g(B^H_s)$ in the interval $[0,t]$  given by
  \[
  S^\nu_n(g,t)=   \sum_{j=0}^{\nt-1}  (B^H_{\frac {j+1} n}-B^H_{\frac jn}) \int_0^1 g\left (B^H_{\frac jn}+\alpha(  B^H_{\frac {j+1} n}-B^H_{\frac jn})\right) \nu(d\alpha),
  \]
  where $n\ge 1$ is an integer and $\lfloor x\rfloor$ denotes the integer part of $x$ for any $x\ge 0$.
  Then, following  \cite{GNRV} \textcolor{black}{and providing the limit exists}, the $\nu$-symmetric integral is defined  as the limit in probability of the $\nu$-symmetric Riemann sums as $n$ tends to infinity, namely,
  \[
    \int_0^t g(B^H_s) d^\nu B^H_s = \lim_{n\rightarrow \infty}  S^\nu_n(g,t).
     \]
 It is  proved in  \cite{GNRV} that  this integral exists for $g=f'$ with
 $f\in \mathcal{C}^{4\ell(\nu)+2}(\R)$, \textcolor{black}{if} the Hurst parameter satisfies $H>\frac 1 {4\ell(\nu)+2}$. Here we denote by  $\ell(\nu)\ge 1$ the \textcolor{black}{largest} positive natural number such that
    \begin{eqnarray}\label{ell}
       \int_0^1\alpha^{2j} \nu(d\alpha) = \frac{1}{2j+1} \qquad \forall j=\textcolor{black}{0,1},\dots,\ell(\nu)-1.
    \end{eqnarray}
 Moreover, \textcolor{black}{in this case}    the integral $\int_0^t f'(B^H_s)  d^\nu B^H_s$  satisfies the chain rule
        \[ f(B^H_t) = f(0) + \int_0^t f'(B^H_s) d^\nu B^H_s. \]
  Basic examples of $\nu$-symmetric Riemann sums and integrals are the following:
  \begin{itemize}
  \item[(i)] If  $\nu= \frac 12 \delta_0 + \frac 12 \delta_1$, then $S^\nu_n$ are the trapezoidal Riemann sums. In  this case $\ell(\nu)=1$ and $\nu$-symmetric integrals exist for $H>\frac 16$.
  \item[(ii)] If  $\nu= \frac 16 \delta _0 + \frac 23 \delta_{1/2} +\frac 16 \delta_1$, then $\ell(\nu)= 2$. In this case, $S^\nu_n$ are the Simpson Riemann sums and  $\nu$-symmetric integrals exist for $H>\frac 1{10}$.
  \item[(iii)] If $\nu$ is the  Lebesgue measure, then $\ell(\nu)=\infty$, and  $\nu$-symmetric integrals exist for any $H\in(0,1)$.
  \end{itemize}

  The lower bound  $\frac 1 {4 \ell(\nu)+2}$ for the Hurst parameter is sharp, in the sense that for $H= \frac 1 {4 \ell(\nu)+2}$ the $\nu$-symmetric integral diverges in $L^2(\Omega)$ for $f(x)=x^2$. This has \textcolor{black}{been} proved, for  the example (i) above, in the references \cite{ChN} and \cite{GNRV}.  The goal of this paper is to show that when $H=\frac 1 {4 \ell(\nu)+2}$, the $\nu$-symmetric Riemann sums converge in distribution and, as a consequence, we obtain a change-of-variable formula  in law with a correction term which is an It\^o stochastic integral with respect to a Brownian motion which is independent of $B^H$.  More precisely, the main result is the following theorem.

  We say that a function  $f:\R\rightarrow\R$  has \textit{moderate growth} if there exist positive constants $A$, $B$  and $\alpha <2$ such that $|f(x)|\leq Ae^{B|x|^\alpha}$ \textcolor{black}{for all $x\in\R$}.

    \begin{thm}  \label{thm1}
  Fix a symmetric  probability measure $\nu$ on $[0,1]$ with $\ell:=\ell(\nu)<\infty$ and let  $B^H=\{B^H_t,t\geq0\}$ be a fractional Brownian motion with Hurst parameter $H=\frac{1}{4\ell+2}$. Consider a function  $f\in \mathcal{C}^{20\ell+5}(\R)$ such that $f$ and its derivatives up to the order  $20\ell+5$ have moderate growth. Then,
        \begin{equation} \label{E1}
     S_n^\nu(f',t) \,\mathrel{\mathop{\longrightarrow}^{\mathrm{\mathcal{L}}}_{\mathrm{n\rightarrow\infty}}} \,      f(B^H_t)-f(0)- c_\nu \int_0^t f^{(2\ell+1)}(B^H_s) \,dW_s,
         \end{equation}
      where $W=\{W_t, t\geq0\}$ is a Brownian motion independent of $B^H$, $c_\nu$ is a constant depending \textcolor{black}{only} on $\nu$, and the convergence holds in the topology of the Skorohod space $D([0,\infty))$.
    \end{thm}
    The value of the constant $c_\nu$ in  (\ref{E1}) is $c_\nu = k_{\nu,\ell} \sigma_\ell$, where $k_{\nu,\ell}$ is defined in (\ref{kanu}) and $\sigma_\ell$ is given by
    \begin{equation} \label{sigma}
    \sigma_\ell^2=1+ 4^{-\ell} \sum_{j=1} ^\infty \left( (j+1)^{ \frac 1{2\ell+1}} +(j-1)^{ \frac 1{2\ell+1}} -2j^{ \frac 1{2\ell+1}} \right)^{2\ell+1} .
    \end{equation}

  The statement of Theorem \ref{thm1} can also be written as the following change-of-variables formula in law:
  \[
    f(B^H_t)= f(0) + \int_0^t f'(B^H_s) d^\nu B^H_s +  c_\nu \int_0^t f^{(2\ell+1)}(B^H_s) \,dW_s.
    \]

    Some particular cases have already been addressed recently in the literature.
  In the  case $\nu=  \frac 12 \delta_0 + \frac 12 \delta_1$ (trapezoidal Riemann sums), the critical value is $H=\frac 16$, and  the \textcolor{black}{corresponding version of Theorem \ref{thm1}} was proved by
  Nourdin, R\'eveillac and Swanson in  \cite{NRS}.  The convergence results for trapezoidal Riemann sums were extended to a general class of Gaussian processes by Harnett and  Nualart in \cite{HN1}.
  If $\nu=\frac16\delta_0+\frac23\delta_{\frac12}+\frac16\delta_1$ (Simpson Riemann sums),  the critical point is $H=\frac{1}{10}$ and the convergence in law for any fixed time $t\ge 0$ was proved by Harnett and Nualart in \cite{HN3}.

  For related results in the case of midpoint Riemann sums, we refer \textcolor{black}{to the works by Burdzy and Swanson  \cite{BS}, Nourdin and R\'eveillac \cite{NR} and Harnett and Nualart \cite{HN2}}.  In this case the critical value of the Hurst parameter is $H=\frac 14$, and the complementary term in the  It\^o formula involves a second derivative.
  \textcolor{black}{See also Nourdin \cite{NourdinJFA}.}

Let us \textcolor{black}{briefly} describe  the main ideas in the proof of  Theorem \ref{thm1}. First, using a decomposition of  $f(B^H_t)-f(0)$ based on Taylor's formula and the properties of the symmetric measure $\nu$, derived in \cite{GNRV}, it turns out that the only nonzero contribution to the limit in law of the $\nu$-symmetric Riemann sums  $ S_n^\nu(f',t) $ is the term
\begin{equation} \label{ws}
 \sum_{j=0}^{\nt-1}      \,f^{(2\ell+1)}(\widetilde{B}^H_{\frac jn }) (\Delta_j^n B^H)^{2\ell+1},
 \end{equation}
 where we used the notation   $\widetilde{B}^H_{\frac jn}=\frac12\big(B^H_{\frac jn}+B^H_{\frac {j+1}n}\big)$ and $\Delta_j^nB^H=B^H_{\frac {j+1} n}-B^H_{\frac jn}$.
 This term is a weighted sum of the odd powers $(\Delta_j^n B^H)^{2\ell+1}$ of the fBm. It is well-known that for $H=\frac{1}{4\ell+2}$, the sums of these odd powers converge in law to a Gaussian random variable. More precisely, the following stable convergence holds
     \begin{equation}  \label{er1}
 \left(    \sum_{j=0}^{\nt -1} \big( \Delta _j^n B^H\big)^{2\ell+1} , B^H_t , t\ge 0\right)\,\mathrel{\mathop{\longrightarrow}^{\mathrm{\mathcal{L}}}_{\mathrm{n\rightarrow\infty}}} \, \left( \sigma_\ell  W_t, B^H_t, t\ge 0 \right),
 \end{equation}
 where $\sigma _\ell$ is defined in (\ref{sigma}) and  in the right-hand side, the \textcolor{black}{process} $W$ is a Brownian motion  independent of $B^H$. The proof of the convergence for a fixed $t$ follows from the \textcolor{black}{Breuer-Major} Theorem
 (we refer to  \cite[Chapter 7]{NouPecBook}   and \cite{CNW} for a proof of this result  based on the Fourth Moment theorem). Then the convergence of the weighted sums (\ref{ws}) follows from the methodology of  small blocks/big blocks  used, for instance in the works \cite{CNW} and \cite{CNP}. However, unlike these references, the convergence to zero of the reminder term cannot be established using fractional calculus techniques because $H<\frac 12$, and it requires the application of integration-by-parts formulas from Malliavin calculus.

 The paper is organized as follows. Section 2 contains some preliminaries on the Malliavin calculus and the fractional Brownian motion.  Section 3 is devoted to the proof of Theorem \ref{thm1} and in Section 4 we show two basic technical lemmas.

\section{Preliminaries}

  In the next two subsections, we discuss some notions of Malliavin calculus and fractional Brownian motion. Throughout the paper $C$ will denote any positive constant, which may change from one expression to another.

\subsection{Elements of Malliavin Calculus}\label{MalliavinCalc}

  Let $\mathfrak{H}$ be a real separable infinite-dimensional Hilbert space and let $X=\{X(h),\,h\in\mathfrak{H}\}$ be an isonormal Gaussian process over $\mathfrak{H}$. This means that $X$ is a centered Gaussian family, defined on a complete probability space $(\Omega ,\mathcal{F},P)$, with a covariance structure given by
    \begin{equation*}
      \E[X(h)X(g)] = \langle h,g \rangle_\mathfrak{H}, \quad h,g\in \mathfrak{H}.
    \end{equation*}
  We assume that $\mathcal{F}$ is the $\sigma$-algebra generated by $X$.

  For any integer $q\geq 1$, let $ \mathfrak{H}^{\otimes q}$ and $\mathfrak{H}^{\odot q}$ denote, respectively, the $q$th tensor product and the $q$th symmetric tensor product of $\mathfrak{H}$.

  Let $\{e_{n},\,n\geq 1\}$ be a complete orthonormal system in $\mathfrak{H}$. Given $f\in \mathfrak{H}^{\odot p}$, $g\in\mathfrak{H}^{\odot q}$ and $r\in\{0,\ldots ,p\wedge q\}$, the $r$th-{\it order contraction} of $f$ and $g$ is the element of $\mathfrak{H}^{\otimes(p+q-2r)}$ defined by
    \begin{equation}\label{contract_oper}
      f\otimes _{r}g=\sum_{i_{1},\ldots ,i_{r}=1}^{\infty }\langle f,e_{i_{1}}\otimes \ldots \otimes e_{i_{r}}\rangle _{\mathfrak{H}^{\otimes r}}\otimes \langle g,e_{i_{1}}\otimes \ldots \otimes e_{i_{r}}\rangle _{\mathfrak{H}^{\otimes r}},
    \end{equation}
  where $f\otimes _{0}g=f\otimes g$ and, for $p=q$, $f\otimes _{q}g=\langle f,g\rangle _{\mathfrak{H}^{\otimes q}}$. Notice that $f\otimes _{r}g$ is not necessarily symmetric. We denote its symmetrization by $f\widetilde{\otimes }_{r}g\in\mathfrak {H}^{\odot (p+q-2r)}$.

  Let $\mathcal{H}_{q}$ denote the $q$th {\it Wiener chaos} of $X$, that is, the closed linear subspace of $L^{2}(\Omega)$
generated by the random variables $\{H_{q}(X(h)),h\in\mathfrak{H},\left\|h\right\| _{\mathfrak{H}}=1\}$, where $H_{q}$ is the $q$th Hermite polynomial defined by
            \begin{equation}  \label{Hermite}
       H_q(x)=(-1)^q e^{x^2/2}\frac{d^q}{dx^q}\big(e^{-x^2/2}\big).
       \end{equation}

  For $q\geq 1$, let $I_q(\cdot)$ denote the generalized Wiener-It\^o multiple stochastic integral. It is known that the map
    \begin{equation}\label{WI_stoch_int}
      I_{q}(h^{\otimes q}) = H_{q}(X(h))
    \end{equation}
  provides a linear isometry between $\mathfrak{H}^{\odot q}$ (equipped with the modified norm $\sqrt{q!}\left\|\cdot\right\|_{\mathfrak{H}^{\otimes q}}$) and $\mathcal{H}_{q}$ (equipped with the $L^2(\Omega)$ norm). For $q=0$, we set by convention $\mathcal{H}_{0}=\R$ and $I_{0}$ equal to the identity map.

   Multiple stochastic integrals satisfy the following product formula.
      Let $p,q\geq1$ positive integers. Let $f\in\mathfrak{H}^{\odot p}$ and $g\in\mathfrak{H}^{\odot q}$. Then,
        \begin{equation} \label{prod}
         I_p(f) I_q(g) = \sum_{z=0}^{p\wedge q} z!{p\choose z}{q\choose z}  I_{p+q-2z}(f\widetilde{\otimes}_z g),
         \end{equation}
      where $\otimes_z$ is the contraction operator defined in (\ref{contract_oper}).

  From the hypercontractivity property of the Ornstein-Uhlenbeck semigroup, it is well-known that all $L^r(\Omega)$-norms, $r>1$, are equivalent on each Wiener chaos. In particular, for any real numberr $r\ge 2$, any integer $p\ge 2$ and any $f\in\mathfrak{H}^{\odot p}$ , we have
  \begin{equation} \label{j5}
  \| I_p(f) \|_{L^r(\Omega)} \le C_{r,p} \| I_p(f) \|_{L^2(\Omega)} = C_{r,p}  \sqrt{p!} \| f\| _{\mathfrak{H} ^{\otimes p}}.
  \end{equation}

  Let $\mathcal{S}$ be the set of all smooth and cylindrical random variables of the form
    \begin{equation*}
      F = g \left(X(\phi_1),\dots,X(\phi_n)\right),
    \end{equation*}
  where $n\geq 1$, $g:\R^n\rightarrow\R$ is an infinitely differentiable function with compact support, and $\phi_i\in\mathfrak{H}$. The {\it Malliavin derivative} of $F$ with respect to $X$ is the element of $L^{2}(\Omega;\mathfrak{H})$ defined as
    \begin{equation*}
      DF = \sum_{i=1}^n \frac{\partial g}{\partial x_i} \left(X(\phi_1),\dots,X(\phi_n)\right) \phi_i.
    \end{equation*}
  By iteration, we can define the $q$th derivative $D^{q}F$ for every $q\geq2$, which is an element of $L^{2}(\Omega;\mathfrak{H}^{\odot q})$.

  For any integer $q\ge 1$ and any real number $p\geq1$, let ${\mathbb{D}}^{q,p}$ denote the closure of $\mathcal{S}$ with respect to the norm $\|\cdot \|_{\mathbb{D}^{q,p}}$, defined as
    \begin{equation*}
      \|F\|_{\mathbb{D}^{q,p}}^{p} = \E\big[|F|^{p}\big] + \sum_{i=1}^{q} \E\left(\|D^{i}F\|_{\mathfrak{H}^{\otimes i}}^{p}\right).
    \end{equation*}
  More generally, for any Hilbert space $V$, we denote by $\mathbb{D}^{q,p}(V)$ the corresponding Sobolev space of $V$-valued random variables.

  The Malliavin derivative $D$ fulfills the following chain rule. If $\varphi:\R^n\rightarrow\R$ is continuously differentiable with bounded partial derivatives and if $F=(F_1,\dots,F_n)$ is a vector of elements of $\mathbb{D}^{1,2}$, then $\varphi(F)\in\mathbb{D}^{1,2}$ and
    \begin{equation*}
      D\varphi(F) = \sum_{i=1}^n \frac{\partial\varphi}{\partial x_i}(F)DF_i.
    \end{equation*}

  We denote by $\delta$ the {\it Skorohod integral}, also called the {\it divergence operator}, which is the adjoint of the operator $D$.  More precisely, a random element $u\in L^{2}(\Omega;\mathfrak{H})$ belongs to the domain of $\delta$, denoted by Dom $\delta$, if and only if, for any $F\in\mathbb{D}^{1,2}$,  we have
    \begin{equation*}
      \big| \E\big( \langle DF,u\rangle_{\mathfrak{H}} \big) \big| \leq c_u \,\|F\|_{L^2(\Omega)},
    \end{equation*}
  where $c_u$ is a constant depending only on $u$. If $u\in\mbox{Dom }\delta$, then the random variable $\delta(u)$ is defined by the duality relationship
    \begin{equation*}
      \E(F\delta(u)) = \E\big( \langle DF,u\rangle_{\mathfrak{H}} \big).
    \end{equation*}
   This is called the Malliavin integration by parts formula and it holds for every $F\in {\mathbb{D}}^{1,2}$. For $q\geq1$, the multiple Skorohod integral is defined iteratively as $\delta^q(u)=\delta(\delta^{q-1}(u))$, with $\delta^0(u)=u$. From this definition, we have
    \begin{equation}\label{dual}
      \E\big( F\delta^q(u) \big) = \E\big( \left\langle D^qF,u\right\rangle_{\mathfrak{H}^{\otimes q}} \big),
    \end{equation}
  for any $u\in\mbox{Dom }\delta^q$ and any $F\in\mathbb{D}^{q,2}$. Moreover, $\delta^q(h)=I_q(h)$ for any $h\in\mathfrak{H}^{\odot q}$.     We refer to \cite{Nbook} for a detailed account on the Malliavin calculus for an arbitrary isonormal Gaussian process.

\subsection{Fractional Brownian motion}

  Let $B^H=\{B^H_t, t\ge 0\}$ denote a fractional Brownian motion with Hurst parameter $H$. Namely, $B^H$ is a centered Gaussian process, defined on a complete  probability space
  $(\Omega, \mathcal{F},P)$ with covariance given by (\ref{cov}).
We assume that $\mathcal{F}$ is generated by $B^H$. Along the paper we suppose that $H< \frac 12$.

  We denote by $\mathcal{E}$ the set of $\R$-valued step functions on $[0, \infty)$. Let $\mathfrak{H}$ be the Hilbert space defined as the closure of $\mathcal{E}$ with respect to the scalar product
    \begin{equation*}
      \left\langle {\mathbf{1}}_{[0,t]},{\mathbf{1}}_{[0,s]}\right\rangle _{\mathfrak{H}} = R(s,t).
    \end{equation*}
  The mapping $\1_{[0,t]}\rightarrow B^H_t$ can be extended to a linear isometry between the Hilbert space $\mathfrak{H}$ and the Gaussian space spanned by $B^H$. In this way $\{B^H(h),h\in\mathfrak{H}\}$ is an isonormal Gaussian process as in Section \ref{MalliavinCalc}.

Recall the notation $\widetilde{B}^H_{\frac jn}=\frac12\big(B^H_{\frac jn}+B^H_{\frac {j+1}n}\big)$ and $\Delta_j^nB^H=B^H_{\frac {j+1} n}-B^H_{\frac jn}$.  Moreover, we  set
 \begin{eqnarray*}
\partial_{\frac jn }&=& \1_{[\frac jn , \frac {j+1} n]}, \\
\varepsilon_t &=& \1_{[0,t]}
\end{eqnarray*}
and
\[
\widetilde{\varepsilon}_{\frac jn  }=\frac12\big(\varepsilon_{\frac  jn }+\varepsilon_{\frac {j+1} n}\big)  =\frac12\big(\1_{[0,\frac jn ]}+\1_{[0,\frac {j+1} n]}\big).
\]

  The  fractional Brownian motion  with Hurst parameter $H$ satifies
  \begin{equation}
\E\big[(\Delta_j^ nB^H)^2\big]=\langle\partial_{\frac jn },\partial_{\frac jn }\rangle_\mathfrak{H}=n^{-2H}. \label{prop_fBM1}
\end{equation}
Moreover,  using the fact that the function $x\rightarrow x^{2H}$ is concave for $H<\frac 12$, for any $t \ge 0$ and any  integer $j\ge 0$, we obtain
\begin{equation}
       \big| \E\big[(\Delta_j^nB^H)B^H_t\big] \big| =  \big| \big\langle\partial_{\frac jn},\varepsilon_t\big\rangle_\mathfrak{H} \big| \leq  n^{-2H}. \label{prop_fBM2}
\end{equation}
The following lemma has been proved in  \cite[Lemma 2.6]{HN3}.
    \begin{lemma} \label{lemma_fBM}
      Let $H<\frac12$ and  let $n\geq2$ be an integer. Then, there exists a constant $C$ not depending on $T$ such that:

      \noindent
     a) For any  $t\in [0,T]$,
     \[
     \sum_{j=0}^{\nT-1} \big| \big \langle \partial_{\frac jn }, \varepsilon_t \big  \rangle  _\mathfrak{H} \big|  \le C \nT^{2H} n^{-2H}.
     \]

     \noindent
     b)  For  any integers $r\geq1$ and  $0\le i\le \nT -1$,
            \begin{equation}  \label{eq4b}
            \sum_{j=0}^{\nT-1} \big| \big\langle\partial_{\frac jn },\partial_{\frac in}\big\rangle_\mathfrak{H}^r \big| \leq Cn^{-2rH},
            \end{equation}
          and consequently
            \begin{equation} \label{eq4a}
             \sum_{j,i=0}^{\nT-1} \big| \big\langle\partial_{\frac jn},\partial_{\frac in}\big\rangle_\mathfrak{H}^r \big|  \leq  C\textcolor{black}{\lfloor nT\rfloor} n^{-2rH}.
             \end{equation}
             \end{lemma}

             The next result  provides useful estimates when we compare two partitions. Its proof is based on computing telescopic sums.
\begin{lemma}
We fix  two integers $n>m\ge 2$,  and for any $j\ge 0$, \textcolor{black}{we} define
  $k:=k(j)= \sup \{i\ge 0: \frac im \le  \frac jn\}$.
The following inequalities hold true for some constant $C$ depending \textcolor{black}{only} on $T$:
\begin{eqnarray}  \label{2.6}
\sum_{j=0}^{\nT-1} \big| \big\langle\partial_{\frac jn},\varepsilon_{\frac {k(j)}m}\big\rangle_\mathfrak{H} \big|  & \leq& Cm^{1-2H}, \\  \label{2.7}
     \sum_{j=0}^{\nT-1} \big| \big\langle\partial_{\frac jn },\widetilde{\varepsilon}_{\frac jn }-\varepsilon_{\frac {k(j)}m}\big\rangle_\mathfrak{H} \big|  &\leq&  C m^{1-2H}
               \end{eqnarray}
               and, for any $0\leq i\leq \nT-1$,
\begin{equation}  \label{2.8}
             \sum_{j=0}^{\nT-1} \big| \big\langle\partial_{\frac jn},\widetilde{\varepsilon}_{\frac in}-\varepsilon_{\frac {k(i)} m} \big\rangle_\mathfrak{H} \big| \leq C m^{-2H} .
            \end{equation}
    \end{lemma}

  \begin{proof}
Let us first show (\ref{2.6}). We can write
    \begin{eqnarray*}
      && \sum_{j=0}^{\nT-1} \Big| \big\langle\partial_{\frac jn}, \varepsilon_{\frac {k(j)}m} \big\rangle_\mathfrak{H} \Big|  =  \sum_{j=0}^{\nT-1} \Big| \E\big[ \big(B^H_{\frac {j+1} n}-B^H_{\frac jn}\big) B^H_{\frac  {k(j)} m }   \big] \Big| \\
      && \qquad  =   \frac12 \sum_{j=0}^{\nT -1} \textcolor{black}{\Bigg|} \left(\frac {j+1} n\right)^{2H} -  \left( \frac jn \right)^{2H} - \left|  \frac {j+1}n- \frac {k(j)} m \right|^{2H} +
      \left| \frac jn -\frac {k(j)} m \right|^{2H} \textcolor{black}{\Bigg|} \\
      && \qquad  \leq \frac12\,  n^{-2H} \sum_{j=0}^{\nT-1} \Big[ (j+1)^{2H}-j^{2H} \Big] \\
      && \hspace{40mm} + \frac12  \sum_{j=0}^{\nT-1}   \left[ \textcolor{black}{\Bigg(} \frac {j+1}n- \frac {k(j)} m \textcolor{black}{\Bigg)^{2H}} -
     \textcolor{black}{\Bigg(} \frac jn -\frac {k(j)} m \textcolor{black}{\Bigg)^{2H}}  \right].
    \end{eqnarray*}
    The first term is a telescopic sum and  it is easy to show that
      $$ \frac12 n^{-2H} \sum_{j=0}^{\nT-1} \left[ (j+1)^{2H}-j^{2H} \right] = \frac12 n^{-2H} (\nT)^{2H}  \leq C \le C m^{1-2H}. $$
    For the second term, observe that, for a fixed $k=0,\dots, \mT+1$,  the sum of the terms for which $k(j)=k$ is telescopic and is bounded by a constant times $\textcolor{black}{n^{-2H}\leq m^{-2H}}$. Summing over all possible values of $k$, we obtain de desired bound $Cm^{1-2H}$.

The inequality (\ref{2.7})  is an immediate consequence of (\ref{2.6}) and the following easy fact:
\textcolor{black}{
\begin{eqnarray*}
 \sum_{j=0}^{\nT-1} \big| \big\langle\partial_{\frac jn },\widetilde{\varepsilon}_{\frac jn } \big\rangle_\mathfrak{H} \big|  &=&\frac12 \sum_{j=0}^{\nT-1} \Big| \E\big[ \big(B^H_{\frac {j+1} n}-B^H_{\frac jn}\big) \big(B^H_{\frac {j+1} n}+B^H_{\frac jn}\big)   \big] \Big|
\\
&=&   \frac{n^{-2H}}2 \sum_{j=0}^{\nT -1} \left[ (j+1)^{2H} -  j^{2H}\right]=\frac12 n^{-2H} (\nT)^{2H}  \leq C \le C m^{1-2H}.
 \end{eqnarray*}
 }

  Let us now proceed \textcolor{black}{with} the proof of   (\ref{2.8}). We can write
    \begin{eqnarray*}
      && \sum_{j=0}^{\nT-1} \Big|\big\langle \partial_{\frac jn},\widetilde{\varepsilon}_{\frac in }- \varepsilon_{\frac {k(i)}m }\big\rangle_\mathfrak{H}\Big| \\
      && \hspace{10mm} = \frac12 \sum_{j=0}^{\nT -1} \Big| \E\Big[ \big(B^H_{\frac {j+1} n }-B^H_{\frac jn }\big) \big(B^H_{\frac in}+B^H_{\frac {i+1} n}-\textcolor{black}{2}\,B^H_{\frac {k(i)}m }\big) \Big] \Big| \\
      && \hspace{10mm} \leq \frac12 \sum_{j=0}^{\nT-1} \Big| \E\Big[ \big(B^H_{\frac {j+1} n}-B^H_{\frac jn }\big) \big(B^H_{\frac in}-B^H_{\frac {k(i)} m}\big) \Big] \Big| \\
      && \hspace{40mm} + \,\frac12 \sum_{j=0}^{\nT-1} \Big| \E\Big[ \big(B^H_{\frac {j+1}n} -B^H_{\frac jn}\big) \big(B^H_{\frac {i+1} n} -B^H_{\frac {k(i)} m}\big) \Big] \Big| \\
      && \hspace{10mm} =: A_1 + A_2.
    \end{eqnarray*}
    For the term $A_1$ we have
\[
      A_1= \frac1 2  \sum_{j=0}^{\nT -1} \left|  H_j \right |,
     \]
      where
      \[
      H_j= \left | \frac in - \frac jn \right|^{2H} -\left |\frac in -\frac {j+1} n \right |^{2H} + \left|  \frac {k(i)} m -\frac {j+1}n \right|^{2H} - \left| \frac {k(i)}m  -\frac jn \right|^{2H}.
      \]
 Taking into account that $\frac {k(i)} m \le \frac in$, it follows that, for $j\ge i$
    \begin{eqnarray*}
      H_j &=& \left (  \frac jn -\frac in  \right)^{2H} -\left ( \frac {j+1} n-\frac in  \right )^{2H} + \left(    \frac {j+1}n-\frac {k(i)} m \right)^{2H} - \left( \frac jn -\frac {k(i)}m\right)^{2H}
       \\
&=& 2H  \int_0^{\frac 1n}  \left[\left( \frac jn +x -\frac {k(i)}m \right)^{2H-1}  -\left( \frac jn +x -\frac in \right)^{2H-1} \right]dx \le 0.
      \end{eqnarray*}
      On the other hand, if $j_0$ is the \textcolor{black}{largest integer}  $j\ge 0$ such that $\frac  {j+1}n \le \frac {k(i)}m$, then, for $j\le j_0$,
         \begin{eqnarray*}
      H_j &=& \left (  \frac in -\frac jn  \right)^{2H} -\left ( \frac in-\frac {j+1}n  \right )^{2H} + \left(   \frac {k(i)} m-  \frac {j+1}n  \right)^{2H} - \left(   \frac {k(i)}m-  \frac {j}n\right)^{2H}
       \\
&=& -2H  \int_0^{\frac 1n}  \left[\left( \frac {k(i)}m- \frac jn -x  \right)^{2H-1}  -\left( \frac in -\frac jn -x   \right)^{2H-1} \right]dx \le 0.
      \end{eqnarray*}
 Consider the decomposition
 \[
 A_1= \frac 12 \left( \sum_{j=0}^{ j_0} \left|  H_j \right | +   \sum_{j=j_0+1}^{ i-1} \left|  H_j \right | + \sum_{j=i}^{ \nT-1} \left|  H_j \right |  \right)
\textcolor{black}{=:} \frac 12 \left( A_{11}+ A_{12} + A_{13} \right).
\]
For the terms $A_{11}$ and $A_{13}$, we obtain, respectively
\begin{eqnarray*}
A_{11}&= &\sum_{j=0}^{ j_0}  (- H_j )  \\
&=& \left ( \frac  in - \frac {j_0+1} n \right)^{2H} - \left( \frac  {k(i)}  m-\frac {j_0+1} n  \right)^{2H}  -\textcolor{black}{\left( \frac in  \right)^{2H}} + \textcolor{black}{\left( \frac {k(i)}m \right) ^{2H}} \\
&\le&  2\left( \frac in -\frac {k(i)}m \right) ^{2H}  \le  Cm^{-2H}
\end{eqnarray*}
and
\begin{eqnarray*}
A_{13} &= &\sum_{j=i}^{ \nT-1}  (- H_j )  \\
&=& \left ( \frac {\nT} n- \frac in \right)^{2H} - \left( \frac {\nT} n -\frac {k(i)} m \right)^{2H} + \left( \frac in -\frac {k(i)}m \right) ^{2H}  \\
&\le& \left( \frac in -\frac {k(i)}m \right) ^{2H}  \le m^{-2H}.
\end{eqnarray*}
    Finally, for the term $A_{12}$, we have
    \begin{eqnarray*}
    A_{12} & \le&   \sum_{j=j_0+1} ^{i-1} \left| \left (  \frac in -\frac jn  \right)^{2H} -\left ( \frac in-\frac {j+1}n  \right )^{2H}  \right| \\
    && + {\color{black}{\sum_{j=j_0+1} ^{i-1}}} \left| \left(    \frac {j+1}n-\frac {k(i)} m \right)^{2H} - \left( \frac jn -\frac {k(i)}m\right)^{2H}  \right|  \\
    &=:& A_{121} + A_{122}.
    \end{eqnarray*}
    The term $A_{121}$ is a telescopic sum which produces a contribution of the form
    \[
     \left( \frac {i-j_0-1}n \right)^{2H} \le Cm^{-2H}
     \]
     and the term $A_{122}$ can be bounded as follows
          \begin{eqnarray*}
     A_{122} &\le&  \left|  \frac {j_0+2}n-\frac {k(i)} m \right|^{2H} + \left| \frac {j_0+1}n -\frac {k(i)}m\right|^{2H} +\sum_{j=j_0+2} ^{i-1} \left[ \left(    \frac {j+1}n-\frac {k(i)} m \right)^{2H} - \left( \frac jn -\frac {k(i)}m\right)^{2H}  \right] \\
     & \le  & Cm^{-2H} +  {\color{black}{\left( \frac i n-\frac {k(i)} m \right)^{2H} - \left( \frac {j_0+2}n -\frac {k(i)}m\right)^{2H} }}   \le Cm^{-2H}.
     \end{eqnarray*}
     \textcolor{black}{The term $A_2$ can be treated in a similar way.}
     This completes the proof.
  \end{proof}

  We will use the following lemma.
  \begin{lemma}  \label{lem21}  For any odd integer $r\ge 1$, we have
      \[
         \big(\Delta_j^n B^H \big)^r = \sum_{u=0}^{\left\lfloor\frac{r}{2}\right\rfloor} C_{r,u} n^{-2uH} \, I_{r-2u}(\partial_{\frac jn}^{\otimes r-2u}),
         \]
         where $C_{r,u}$  are some integers.
         \end{lemma}
         \begin{proof}
      By (\ref{prop_fBM1}), we have $\|\Delta_j^nB^H\|_{L^2(\Omega)}=n^{-H}$. For  any integer $q\geq1$, we  recall (see (\ref{Hermite})) that $H_q(x)$ denotes  the Hermite polynomial of degree $q$. \textcolor{black}{Using} an inductive argument \textcolor{black}{coming} from the relation $H_{q+1}(x)=xH_q(x)-qH_{q-1}(x)$, it follows that
        \begin{equation} \label{eq2a}
         x^r = \sum_{u=0}^{\left\lfloor\frac{r}{2}\right\rfloor} C_{r,u} H_{r-2u}(x),
         \end{equation}
    where $C_{r,u}$ is an integer. Applying (\ref{WI_stoch_int}) to   $h=n^H\partial_{\frac jn}$, that is, $X(h)
    ={\Delta_j^nB^H}/{\|\Delta_j^nB^H\|_{L^2(\Omega)}}=n^H\Delta_j^nB^H$, we can write
        \begin{equation} \label{eq3a}
         H_r(n^H\Delta_j^nB^H) = I_r(n^{rH}\partial_{\frac jn}^{\otimes r}).
         \end{equation}
Substituting (\ref{eq3a}) into (\ref{eq2a}), yields
        \[
         n^{rH}\big(\Delta_j^nB^H \big)^r = \sum_{u=0}^{\left\lfloor\frac{r}{2}\right\rfloor}  C_{r,u}  I_{r-2u}(n^{(r-2u)H}\partial_{\frac jn}^{\otimes r-2u}),
         \]
    which implies the desired result.
\end{proof}

\section{Proof of Theorem \ref{thm1}}
We recall that  $\ell=\ell(\nu)$ \textcolor{black}{is defined by (\ref{ell}).}
The  first ingredient of the proof is the following development,  established in \cite{GNRV}, based on Taylor's formula and the properties of the measure $\nu$
\begin{eqnarray}   \nonumber
f(b)&=& f(a) +(b-a) \int_0^1 f'(a+ \alpha(b-a)) \nu(d\alpha) \\   \label{e1}
&&+ \sum_{h=\ell} ^{2\ell}  k_{\nu,h}  f^{\textcolor{black}{(}2h+1\textcolor{black}{)}} \left( \frac {a+b}2 \right) (b-a) ^{2h+1} + (b-a) ^{4\ell+2} C(a,b),
\end{eqnarray}
where $a,b\in \R$ and  $C(a,b)$ is a continuous function such that $C(a,a)=0$. The constants $k_{\nu,h}$ are given by
\begin{equation}  \label{kanu}
k_{\nu,h} = \frac  1{(2h )!} \left[ \frac 1{(2h+1) 4^{h }} -\int_0^1 \left(\alpha -\frac 12 \right) ^{2 h} \nu(d\alpha)\right].
\end{equation}
Applying equality  (\ref{e1}) to $a= B^H_{\frac jn}$ and $b= B^H_{\frac {j+1} n}$ and using the notation
$\widetilde{B}^H_{\frac jn}=\frac12\big(B^H_{\frac jn}+B^H_{\frac {j+1} n}\big)$ and
       $\Delta_j^nB^H=B^H_{\frac {j+1} n}-B^H_{\frac jn}$, yields
  \begin{eqnarray*}
    f(B^H_{\nt})-f(0) &=& \sum_{j=0}^{\nt-1} \Delta_j ^nB^H \int_0^1 f'\big(B^H_{\frac jn }+\alpha\Delta _j^n B^H\big) \nu(d\alpha)  \\
    && +\sum_{h=\ell }^{2\ell} \sum_{j=0}^{\nt-1}    k_{\nu,h} \,f^{(2h+1)}(\widetilde{B}^H_{\frac jn }) (\Delta_j^n B^H)^{2h+1}  \\
    &&   + \sum_{j=0}^{\nt-1} C(B^H_{\frac jn},B^H_{\frac {j+1} n}) (\Delta_j^nB^H)^{4\ell+2},
  \end{eqnarray*}
  which can be written as
    \begin{equation}   \label{eqprinc2}
    f(B^H_{\nt})-f(0) = S^\nu_n(f',t) + \sum_{h=\ell}^{2\ell} \Phi_n^h(t) + R_n(t),
    \end{equation}
    where, for each $h= \ell ,\dots, 2\ell$,
    \begin{equation}   \label{phih}
    \Phi_n^h(t)= \sum_{j=0}^{\nt-1}    k_{\nu,h} \,f^{(2h+1)}(\widetilde{B}^H_{\frac jn }) (\Delta_j^n B^H)^{2h+1}
    \end{equation}
    and
    \[
    R_n(t)=\sum_{j=0}^{\nt-1} C(B^H_{\frac jn},B^H_{\frac {j+1} n}) (\Delta_j^nB^H)^{4\ell+2}.
    \]
Let us consider the convergence of each  term in the decomposition (\ref{eqprinc2}). First we  will show that
the term $ R_n(t)$ converges to zero  in probability, uniformly in compact sets. In fact, for any $T>0$, $K,\epsilon >0$, we can write
\begin{equation} \label{e2}
P\left( \sup_{0\le t\le T} | R_n(t)| >\epsilon \right) \le P \left({\sup_{s,t \in [0,T] \atop |t-s| \le \frac 1n}} |C(B^H_s,B^H_t)| > \frac 1K \right)
+ P \left(  \sum_{j=0}^{\nT-1} (\Delta_j^nB^H)^{4\ell+2} > K \epsilon \right).
   \end{equation}
   Taking into account that $H=\frac 1{ 4\ell+2}$ and using (\ref{prop_fBM1}), we can write, \textcolor{black}{with $\mu_k$ denoting the $k$th moment of the standard Gaussian},
 \begin{equation} \label{e3}
   P \left(  \sum_{j=0}^{\nT-1} (\Delta_j^nB^H)^{4\ell+2} > K \epsilon \right) \le \frac  {\textcolor{black}{\mu_{4\ell+2}}}{K \epsilon} \frac{\nT}n \le \frac   {T\textcolor{black}{\mu_{4\ell+2}}}{K\epsilon}.
      \end{equation}
   From (\ref{e2}) and (\ref{e3}), letting first $ n\rightarrow \infty$ and then $K\rightarrow \infty$ it follows that for any $\epsilon>0$ and $T>0$,
   \[
   \lim_{n\rightarrow \infty }P\left( \sup_{0\le t\le T} | R_n(t)| >\epsilon \right)  =0.
   \]

   On the other hand, by Lemma \ref{lemphi}, the terms $\Phi^h_n$ with $h= \ell+1 ,\dots, 2\ell$ converge to zero in the topology of $D([0,\infty))$ and do not contribute to the limit.
  As a consequence, the proof of Theorem \ref{thm1} follows from the next proposition.
  \begin{prop} \label{prop1}
  Under the assumptions of Theorem  \ref{thm1}, \textcolor{black}{one has}
      \begin{equation}  \label{e5}
\Phi_n^{\ell }(t)=\sum_{j=0}^{\nt-1}      \,f^{(2\ell+1)}(\widetilde{B}^H_{\frac jn }) (\Delta_j^n B^H)^{2\ell+1}  \,\mathrel{\mathop{\longrightarrow}^{\mathrm{\mathcal{L}}}_{\mathrm{n\rightarrow\infty}}} \sigma_{\ell} \int_0^t f^{(2\textcolor{black}{\ell}+1)}(B^H_s) \,dW_s,
         \end{equation}
      where $W=\{W_t, t\geq0\}$ is a Brownian motion independent of $B^H$, $\sigma_{\ell}$ is the constant  defined in (\ref{sigma}), and the convergence holds in the topology of the Skorohod space $D([0,\infty))$.
 \end{prop}

\begin{proof}
  In order to show Proposition \ref{prop1}, we will \textcolor{black}{first} prove that the sequence of processes $\{ \Phi_n^\ell(t), t\ge 0\}$
 is tight in  $D([0,\infty))$, and \textcolor{black}{then that} their finite dimensional \textcolor{black}{distributions} converge to those of
 \[
 \left\{\sigma_{\ell} \int_0^t f^{(2\textcolor{black}{\ell}+1)}(B^H_s) \,dW_s, t\ge 0\right\}.
 \]
 Notice that the tightness of the sequence $\Phi^\ell _n$ is a consequence of Lemma \ref{lemphi}. \textcolor{black}{Indeed}, this lemma implies that for any $0\le s<t\le T$, there exist a constant depending on $T$, such that
   \[
   \E \left[ |\Phi_n^\ell(t)- \Phi_n^\ell(s) |^4 \right] \le  C  \sum_{N=2}^4 \left(  \frac {\nt -\ns }n \right)^ {N}.
   \]

   It remains to show the convergence of the finite-dimensional distributions.
 Fix a finite set of points $0\le t_1 < \cdots \le  t_d \le T$. We want to show  the \textcolor{black}{following convergence} in law, as $n$ tends to infinity:
 \begin{equation} \label{conV}
( \Phi_n^\ell(t_1), \dots, \Phi_n^\ell (t_d))  \mathrel{\mathop{\longrightarrow}^{\mathrm{\mathcal{L}}}_{\mathrm{n\rightarrow\infty}}} (Y_1, \dots, Y_d),
\end{equation}
where
\[
Y_i= \sigma_{\ell} \int_0^{t_i} f^{(2\textcolor{black}{\ell}+1)}(B^H_s) \,dW_s, \quad i=1,\dots, d,
\]
     $W=\{W_t, t\geq0\}$ is a Brownian motion independent of $B^H$, \textcolor{black}{and} $\sigma_{\ell}$ is the constant  defined in (\ref{sigma}).

 Taking into account the convergence (\ref{er1}), the main ingredient in the proof of the convergence (\ref{conV}) is the methodology based on the small blocks/big blocks (see, for instance,   \cite{CNW}).
  This method  consists in considering  two integers $2\le m <n$ and let first $n$ tend to infinity and later $m$ tend to infinity.  For any $k\ge 0$ we define
 the set $$
 \textcolor{black}{
 I_k= \{ j\in \{0,\ldots, \nti-1\}:  \frac km \le  \frac jn  < \frac {k+1}m\}.
 }
 $$
  The basic ingredient in this approach is the  decomposition
  \begin{eqnarray*}
   \Phi_n^\ell(t_i) &=& \textcolor{black}{ \sum_{k=0}^{ {\color{black}{\mti}} } \sum_{j\in I_k} f^{(2\ell +1)}(B^H_{\frac km }) \big(\Delta_j^n B^H\big)^{2\ell+1}} \\
    && \textcolor{black}{+ \sum_{k=0}^{ {\color{black}{\mti}} }\sum_{j\in I_k} \Big[f^{(2\ell+1)}(\widetilde{B}^H_{\frac jn })-f^{(2\ell +1)}(B^H_{\frac km})\Big] \big(\Delta_j^n B^H\big)^{2\ell+1}} \\
    &=:& A_{n,m}^{(1,i)} + A_{n,m}^{(2,i)}.
  \end{eqnarray*}
  From Lemma \ref{lem1} with $r=2\ell +1$ and $\phi= f^{(2\ell +1)}$, we can write, for any $q>2$,
  \begin{eqnarray*}
  \E[  (A_{n,m}^{(2,i)})^2 ]& \le&  C  \sup_{0\leq w\leq 3(2\ell +1)} \sup_{0\leq j\leq \nT-1} \big\|  f^ {(w)}({\color{black}{ \widetilde{B}^H_{\frac jn} }})-f^ {(w)}({\color{black}{ B^H_{\frac {k(j)}m} }})\big\|_{L^q(\Omega)}   ^2   \\
          &&   + \textcolor{black}{C} \sup_{0\leq w\leq 3(2\ell+1) } \sup_{0\leq j\leq \nT-1}\big\| f^{(w)}(\widetilde{B}^H_{\frac jn })\big\|_{L^2(\Omega)}   ^2 (m^{-2H} + n^{2H-1} \,m^{2-4H}) \\
          &&  + \textcolor{black}{C} \sup_{0\leq w\leq 3(2\ell+1)} \sup_{0\leq i,j\leq \nT-1} \big\|  f^{(w)}(\widetilde{B}^H_{\frac in })\big\|_{L^2(\Omega)}\big\|f^{(w)}(\widetilde{B}^H_{\frac jn})- f^{(w)}({\color{black}{ B^H_{\frac {k(j)}m } }})\big\|_{L^2(\Omega)} \\
          &&  \quad \times (1 + n^{2H-1} m^{2-4H}) \\
          &\le&  C\textcolor{black}{\Bigg[}  m^{-2H} +n^{2H-1} m^{2-4H}  \\
          && \textcolor{black}{\times\bigg( 1+} \sup_{0\leq w\leq 3(2\ell+1) }
           \sup_{s,t \in [0,T] \atop |t-s|\le \frac 1m}     \textcolor{black}{ \bigg\|} f^{(w)}(\frac {B^H_s+ B^H_{s+\frac 1n}}2)   -f^ {(w)}(B^H_t)    \textcolor{black}{\bigg\|^2_{L^q(\Omega)}}   \textcolor{black}{\bigg) \Bigg]},
  \end{eqnarray*}
  where   $k:=k(j)= \sup \{i\ge 0: \frac im \le  \frac jn\}$.
  This implies
  \[
  \limsup_{\textcolor{black}{n}\rightarrow \infty}    \E[  (A_{n,m}^{(2,i)})^2    ]  \le  C\Big(  m^{-2H}
           + \sup_{0\leq w\leq 3(2\ell+1) }
           \sup_{s,t \in [0,T] \atop |t-s|\le \frac 1m}      \big\| f^{(w)}(B^H_s)   -f^ {(w)}(B^H_t)    \big\|_{L^2(\Omega)}  ^2   \Big),
\]
which converges to zero as $\textcolor{black}{m}$ tends to infinity.

On the other hand,  from (\ref{er1}) we deduce that the  vector $( A_{n,m}^{(1,1)}, \dots A_{n,m}^{(1,d)}) $ converges in law,  as $\textcolor{black}{n}$ tends to infinity,  to  the vector with components
        \[
\sigma_\ell   \sum_{\textcolor{black}{k}=0}^{ {\color{black}{\mti}} } f^{(2\ell+1)}(\textcolor{black}{B}^H_{\textcolor{black}{\frac km}})  ( W_{\frac {\textcolor{black}{k}+1}{\textcolor{black}{m}}}- W_{\frac{\textcolor{black}{k}}{\textcolor{black}{m}}} ),
        \]
        $i=1,\dots, d$, where $W$ is a Brownian motion independent of $B^H$.
        Each of these components converges in   $L^2(\Omega)$ to the stochastic integral  $ \sigma_\ell \int_0^{t_i}  f^{(2\ell +1)}(B^H_s) \,dW_s$, as $\textcolor{black}{m}$ tends to infinity.
        This completes the proof of the theorem.
  \end{proof}

\section{Appendix}

  This section is devoted to state and prove  a couple of technical  lemmas. The first lemma is the basic ingredient to show that the sequence of processes $\Phi_n^\ell$ are tight and the processes $\Phi^{h}_n$ for $h=\ell+1, \dots, 2\ell$ converge to zero in $D([0,\infty))$. For this  we need to estimate the \textcolor{black}{fourth} moment of the increments of the processes $\Phi_n^h$.

  \begin{lemma}  \label{lemphi}
  Consider the processes $\Phi_n^h$, $h=\ell, \dots, 2\ell$ defined in (\ref{phih}). Then, for any $0\le s<t\le T$  we have
  \[
   \E \left[ |\Phi_n^h(t)- \Phi_n^h(s) |^4 \right] \le  C  \sum_{N=2}^4 \left( \nt -\ns  \right)^ {N}   n^{-2NH(2h+1)},
   \]
   where the constant $C$ depends \textcolor{black}{only} on $T$.
  \end{lemma}

  \begin{proof}
     For any $0\le s<t\le T$ we can write
\[
 \E \left[ |\Phi_n(t)- \Phi_n(s) |^4 \right] =  \sum_{j_1,j_2,j_3,j_4 =\ns} ^{\nt-1}
\E \left[ \prod_{i=1}^4   \left( f  ^{(2h+1)}(\widetilde{B}^H_{\frac {j_i}n }) (\Delta_{j_i}^n B^H)^{2h+1} \right) \right].
\]
By Lemma \ref{lem21} we obtain
\[
(\Delta_{j_i}^n B^H)^{2h+1}=\sum_{u=0}^{h} C_{2h+1,u} n^{-2uH} \, I_{2h+1-2u}(\partial_{\frac {j_i}n}^{\otimes  (2h +1-2u)  } ),
\]
which leads to
\[
 \prod_{i=1}^4  (\Delta_{j_i}^n B^H)^{2h+1}=  \sum_{u_1,u_2,u_3,u_4=0}^{h}  C_{h, \bf{u}} n^{-2 | {\bf u}|H} \prod _{i=1}^4  \left( \, I_{2h+1-2u_i}(\partial_{\frac {j_i}n}^{\otimes  (2h +1-2u_i)  } ) \right),
 \]
 where $C_{h ,\bf{u}}$ is a constant depending on $h$ and the vector ${\bf u}= (u_1, u_2, u_3, u_4)$ and we use the notation  $|{\bf u}| =u_1+u_2+u_3+u_4$.  To simplify the notation we write
 $2h +1 -u_i=  v_i$ for $i=1,2,3,4$. The product formula for multiple stochastic integrals allows us to write
 \begin{eqnarray*}
&&  \prod _{i=1}^4  \left( \, I_{v_i}(\partial_{\frac {j_i}n}^{\otimes  v_i  } ) \right) = \sum_{\alpha \in \Lambda}   C_\alpha
\prod_{1\le i<k \le 4} \langle \partial _{\frac {j_i} n}, \partial_ {\frac {j_k}n} \rangle_{\mathcal{H}}^{\alpha_{ik}}\\
    && \quad  \times  I_{ |{\bf v}| -2 |\alpha|}
    \left(\partial_{\frac {j_1}n} ^{\otimes ^{v_1- \alpha_{12}- \alpha_{13} -\alpha_{14} }}  \otimes    \partial_{\frac {j_2}n} ^{\otimes ^{v_{2}- \alpha_{12}- \alpha_{23} -\alpha_{24} }}  \otimes
    \partial_{\frac {j_3}n} ^{\otimes ^{v_{3}- \alpha_{13}- \alpha_{23} -\alpha_{34} }}  \otimes \partial_{\frac {j_4}n} ^{\otimes ^{v_4- \alpha_{14}- \alpha_{24} -\alpha_{34} }}
     \right),
  \end{eqnarray*}
 where $|{\bf v}| =v_1+v_2+v_3+v_4= 8h+4 - |{\bf u}|$,   $\Lambda$ is the set of all  multi\textcolor{black}{indices} $ \alpha=(\alpha_{12}, \alpha_{13}, \alpha_{14}, \alpha_{23}, \alpha_{24},\alpha_{34})$ with $\alpha_{ik} \ge 0$, such that
 \begin{eqnarray*}
 \alpha_{12}+ \alpha_{13} +\alpha_{14}  &\le& v_1 \\
  \alpha_{12}+ \alpha_{23} +\alpha_{24}  &\le& v_2 \\
   \alpha_{13}+ \alpha_{23} +\alpha_{34}  &\le& v_3 \\
    \alpha_{14}+ \alpha_{24} +\alpha_{34}  &\le& v_4.
 \end{eqnarray*}
 For any ${\bf j}= (j_1, j_2, j_3, j_4)$, $\ns \le j_i \le \nt-1$,  we set
 \[
 Y_{\bf j } =\prod_{i=1}^4   f  ^{(2h+1)}(\widetilde{B}^H_{\frac {j_i}n }),
 \]
 and
\[
h_{ {\bf j}, {\bf \alpha}, {\bf v}} =\partial_{\frac {j_1}n} ^{\otimes ^{v_1- \alpha_{12}- \alpha_{13} -\alpha_{14} }}  \otimes    \partial_{\frac {j_2}n} ^{\otimes ^{v_{2}- \alpha_{12}- \alpha_{23} -\alpha_{24} }}  \otimes
    \partial_{\frac {j_3}n} ^{\otimes ^{v_{3}- \alpha_{13}- \alpha_{23} -\alpha_{34} }}  \otimes \partial_{\frac {j_4}n} ^{\otimes ^{v_4}- \alpha_{14}- \alpha_{24} -\alpha_{34} }.
    \]
 Applying the duality formula  (\ref{dual}) we obtain
 \[
 \E \left[ Y_{\bf j }     I_{ |{\bf v}| -2 |\alpha|} (h_{ {\bf j}, {\bf \alpha}, {\bf v}})\right]=
 \E \left[  \langle D^{ |{\bf v}| -2 |\alpha|} Y_{\bf j } , h_{ {\bf j}, {\bf \alpha}, {\bf v}} \rangle_{{\mathcal H}^{ \otimes  |{\bf v}| -2 |\alpha|}} \right].
     \]
     Therefore, we have shown the following formula
     \begin{eqnarray*}
      \E \left[ |\Phi_n^h(t)- \Phi_n^h(s) |^4 \right] &=&  \sum_{ {\bf j} }
      \sum_{\bf u}   C_{h, \bf{u}} n^{-2 | {\bf u}|H}
      \sum_{\alpha \in \Lambda}   C_\alpha   \\
      &&\times \left(
      \prod_{1\le i<k \le 4} \langle \partial _{\frac {j_i} n}, \partial_ {\frac {j_k}n} \rangle_{\mathcal{H}}^{\alpha_{ik}} \right)
      \E \left[  \langle D^{ |{\bf v}| -2 |\alpha|} Y_{\bf j } , h_{ {\bf j}, {\bf \alpha}, {\bf v}} \rangle_{{\mathcal H}^{ \otimes  |{\bf v}| -2 |\alpha|}} \right],
      \end{eqnarray*}
      where the components of ${\bf j}$ satisfy $\ns \le j_i \le \nt-1$ and  $0\le u_i \le  h$. Finally, the inner product
      $\langle D^{ |{\bf v}| -2 |\alpha|} Y_{\bf j } , h_{ {\bf j}, {\bf \alpha}, {\bf v}} \rangle_{{\mathcal H}^{ \otimes  |{\bf v}| -2 |\alpha|}}$ can be expressed in the form
      \[
      \sum_{\beta \in \Gamma}  \Phi_\beta \prod_{1\le i,k \le 4}  \langle \partial _{\frac {j_i} n},  \widetilde{\varepsilon}_ {\frac {j_k}n} \rangle_{\mathcal{H}}^{\beta_{ik}},
      \]
      where $\beta=(\beta_{ik})_{1\le i,k \le 4}$ is a matrix with nonnegative entries such that
      \begin{eqnarray*}
      \sum_{k=1}^4 \beta_{1k} &=&   v_1- \alpha_{12}- \alpha_{13} -\alpha_{14}  \\
           \sum_{k=1}^4 \beta_{2k} &=&   v_2- \alpha_{12}- \alpha_{23} -\alpha_{24}\\
                 \sum_{k=1}^4 \beta_{3k} &=&   v_3- \alpha_{13}- \alpha_{23} -\alpha_{34}\\
                       \sum_{k=1}^4 \beta_{4k} &=&   v_4- \alpha_{14}- \alpha_{24} -\alpha_{34} .
      \end{eqnarray*}
      Notice that
         $|\beta|= \sum_{i,k=1}^4 \beta_{ik}= |{\bf v}| -2|\alpha|$. The random variables $\Phi_\beta$ are linear combination of products of the form $\prod_{i=1}^4 f^{(w_i)} (B^H_{\widetilde {\varepsilon}_{\frac {j_i}n}})$, with $2h+1 \le  w_i \le  2h+1+|{\bf v}| -2|\alpha|$.
      This leads to the following estimate
      \[
       \E \left[ |\Phi_n^h(t)- \Phi_n^h(s) |^4 \right]\le   C \sum_{ {\bf j} } \sum_{\bf u}  n^{-2 | {\bf u}|H}   \sum_{\alpha \in \Lambda}    \sum_{\beta \in \Gamma}   \prod_{1\le i<k \le 4}  \left|  \langle \partial _{\frac {j_i} n}, \partial_ {\frac {j_k}n} \rangle_{\mathcal{H}}^{\alpha_{ik}}  \right| \prod_{1\le i,k \le 4}   \left| \langle \partial _{\frac {j_i} n},  \widetilde{\varepsilon}_ {\frac {j_k}n} \rangle_{\mathcal{H}}^{\beta_{ik}} \right|.
       \]
      Consider the decomposition of the above sum as follows
      \[
      \E \left[ |\Phi_n(t)- \Phi_n(s) |^4 \right]\le   C \left( A_n^{(1)} + A_n^{(2)} + A_n^{(3}) \right),
      \]
      where $A_n^{(1)}$ contains all the terms such that  at least two components of $\alpha$ are nonzero,
      $A_n^{(2 )}$ contains all the terms such that  one component of $\alpha$ is nonzero and the others vanish, and  $ A_n^{(3\textcolor{black}{)}}$ contains all the terms such that all the components of $\alpha$ are zero.

      \medskip
      \noindent
      {\it Step 1.} Let us first estimate $A_n^{(1)}$. Without any loss of generality, we can assume that   $\alpha_{12}  \ge 1$ and  $\alpha_{13} \ge 1$.
     From (\ref{eq4b}) with $r=1$, we obtain
      \begin{equation} \label{m1}
      \sum_{j_1=\ns}^ {\nt-1}   \left| \langle \partial _{\frac {j_1} n}, \partial_ {\frac {j_2}n} \rangle_{\mathcal{H}}  \right|\le C n^{-2H}
      \end{equation}
      and
          \[
      \sum_{j_3=\ns} ^{\nt-1}   \left| \langle \partial _{\frac {j_1} n}, \partial_ {\frac {j_3}n} \rangle_{\mathcal{H}}  \right|\le C n^{-2H}.
      \]
      We estimate each of the remaining factors by $n^{-2H}$. In this way, we obtain a bound of the form
      \[
      A_n^{(1)} \le C   ( \nt- \ns) ^{2} n^{-2H( |{\bf u}| +|\alpha| +|\beta| )}.
      \]
      Taking into account that $|\alpha| \le \frac 12 |{\bf v}|$,
      \begin{eqnarray*}
       |{\bf u}| +|\alpha| +|\beta|  &=& |{\bf u}| +|\alpha| + |{\bf v}| -2 |\alpha|\\
       & =&  |{\bf u}| + |{\bf v}| - |\alpha| \ge  |{\bf u}|  +\frac { |{\bf v}| }2= 4h +2 -\frac  {|{\bf u}| }2 \ge 4h +2,
       \end{eqnarray*}
       and, as a consequence,
       \begin{equation} \label{t1}
        A_n^{(1)} \le C   ( \nt- \ns) ^{2} n^{-4H(2h+1)}.
        \end{equation}

            \medskip
      \noindent
      {\it Step 2.}   For the term $A_n^{(2)}$,  we can assume that $\alpha_{12} \ge 1$ and all the other components of $\alpha$ vanish. In this case, we still have the inequality (\ref{m1}).  Then, we estimate each of the remaining factors by $n^{-2H}$. In this way, we obtain a bound of the form
      \[
      A_n^{(2)} \le C   ( \nt- \ns) ^{3} n^{-2H( |{\bf u}| +|\alpha| +|\beta| )}.
      \]
      Taking into account that $|\alpha|  = \alpha_{12}  \le v_1 =2h +1 -u_1 \le 2h +1$,
      \begin{eqnarray*}
       |{\bf u}| +|\alpha| +|\beta|  &=& |{\bf u}| +|\alpha| + |{\bf v}| -2 |\alpha|\\
       & =&  |{\bf u}| + |{\bf v}| - |\alpha| \ge  |{\bf u}|  + |{\bf v}|  -2h -1= 6h +3,
       \end{eqnarray*}
       and, as a consequence,  we obtain
       \begin{equation} \label{t2}
        A_n^{(2)} \le C   ( \nt- \ns) ^{3} n^{-6H(2h+1)}.
        \end{equation}

      \medskip
      \noindent
      {\it Step 3.}   Estimating all terms by $n^{-2H}$, we get
      \[
      A_n^{(3)} \le C   ( \nt- \ns) ^{4} n^{-2H( |{\bf u}| +|\beta| )}.
      \]
    We have
\[
       |{\bf u}|  +|\beta|  = |{\bf u}| + |{\bf v}|  =  8h +4,
       \]
       and, as a consequence,  we obtain
       \begin{equation} \label{t3}
        A_n^{(3)} \le C   ( \nt- \ns) ^{4} n^{-8H(2h+1)}.
        \end{equation}

In conclusion, from (\ref{t1}), (\ref{t2}) and (\ref{t3}), we obtain the desired estimate.
 This completes the proof of the lemma.
 \end{proof}

  The second lemma  provides a bound for  the residual term in the application of the small blocks / big blocks technique and it is a variation of   \textcolor{black}{\cite[Lemma 3.2]{HN3}}. Its proof is based on the techniques of Malliavin calculus.
  As before  for two integers $n>m\ge 2$, for any $j\ge 0$ we define
  $k:=k(j)= \sup \{i\ge 0: \frac im \le  \frac jn\}$.

    \begin{lemma}  \label{lem1}
      Let $r=1,3,5,\dots$ and  $n>m\geq2$ be two integers. Let $\phi:\R\rightarrow\R$ be a $\mathcal{C}^{2r}$ function such that $\phi$ and all derivatives up to order $2r$ have moderate growth, and let  $B^H=\{B^H_t,t\geq0\}$ be a  fBm with Hurst parameter $H<\frac12$. Then,  for any $q>2$ and any $T>0$,
        \begin{eqnarray*}
          &&\sup_{t\in [0,T]} \E\left[ \left( \sum_{j=0}^{\nt -1}   \left(\phi(\widetilde{B}^H_{\frac jn})-\phi(B_{\frac {k(j)}m })\right) (\Delta_j^nB^H)^r \right)^2 \right] \leq C \varGamma_{m,n}   n^{1-2rH},
        \end{eqnarray*}
      where   $C$ is a positive constant depending on $q$, $r$, $H$ and $T$, and
        \begin{eqnarray*}
          \varGamma_{m,n} &:=& \sup_{0\leq w\leq 2r} \sup_{0\leq j\leq \nT-1} \big\| \phi^ {(w)}(\widetilde{B}^H_{\frac jn})-\phi^ {(w)}(B^H_{\frac {k(j)}m} )\big\|_{L^q(\Omega)}   ^2   \\
          &&   + \sup_{0\leq w\leq 2r} \sup_{0\leq j\leq \nT-1}\big\|\phi^{(w)}(\widetilde{B}^H_{\frac jn })\big\|_{L^2(\Omega)}   ^2 (m^{-2H} + n^{2H-1} \,m^{2-4H}) \\
          &&  + \sup_{0\leq w\leq 2r} \sup_{0\leq i,j\leq \nT-1} \big\|\phi^{(w)}(\widetilde{B}^H_{\frac in })\big\|_{L^2(\Omega)}\big\|\phi^{(w)}(\widetilde{B}^H_{\frac jn})-\phi^{(w)}(B^H_{\frac {k(j)}n })\big\|_{L^2(\Omega)} \\
          &&  \quad \times (1 + n^{2H-1} m^{2-4H}).
        \end{eqnarray*}
          \end{lemma}

  \begin{proof}
  The proof is based in the methodology used to show Lemma 3.2 in \cite{HN3}.
    To simplify notation, let $Y_j (\phi):=\phi(\widetilde{B}^H_{\frac jn })-\phi(B^H_{\frac {k(j)}m})$, and set
    \[
    I_t:= \E\left[ \left( \sum_{j=0}^{\nt-1} Y_j(\phi) \big(\Delta_j^nB^H\big)^r \right)^2 \right] .
    \]
       From Lemma \ref{lem21} we obtain
    \begin{eqnarray*}
I_t&=&  \sum_{u,v=0}^{\left\lfloor\frac{r}{2}\right\rfloor} C_{r,u}C_{r,v} \,n^{-2H(u+v)} \sum_{i,j=0}^{\nt-1} \E\Big[ Y_i(\phi) Y_j(\phi) \, I_{r-2u}(\partial_{\frac in}^{\otimes r-2u}) \, I_{r-2v}(\partial_{\frac jn}^{\otimes r-2v}) \Big] \\
      & \leq & C \sum_{u,v=0}^{\left\lfloor\frac{r}{2}\right\rfloor}   \,n^{-2H(u+v)} \sum_{i,j=0}^{\nT-1} \Big| \E\Big[ Y_i(\phi) Y_j(\phi) \, I_{r-2u}(\partial_{\frac in}^{\otimes r-2u}) \,I_{r-2v}(\partial_{\frac jn}^{\otimes r-2v}) \Big] \Big|.
      \end{eqnarray*}
      Then we apply  the product formuila (\ref{prod}) in order to develop the product of two multiple divergence operators and we end up with
      \begin{eqnarray*}
    I_t &\le &       C \sum_{u,v=0}^{\left\lfloor\frac{r}{2}\right\rfloor}
      \sum_{z=0}^{(r-2u)\wedge(r-2v)}
      n^{-2H(u+v)} \sum_{i,j=0}^{\nT-1} \Big| \E\Big[ Y_i(\phi) Y_j(\phi) \\
      && \times  I_{2r-2(u+v)-2z} \big(\partial_{\frac in }^{\otimes r-2u-z} \widetilde{\otimes} \partial_{\frac jn}^{\otimes r-2v-z}\big) \big\langle\partial_{\frac in},\partial_{\frac jn}\big\rangle_\mathfrak{H}^z \Big] \Big| \\
      & = & C \sum_{u,v=0}^{\left\lfloor\frac{r}{2}\right\rfloor} n^{-2H(u+v)} \sum_{i,j=0}^{\nT-1} \Big| \E\Big[ Y_i (\phi)Y_j(\phi) \, I_{2r-2(u+v)} \big(\partial_{\frac in}^{\otimes r-2u} \widetilde{\otimes} \partial_{\frac jn }^{\otimes r-2v}\big) \Big] \Big| \\
      &&+C
      \sum_{u,v=0}^{\left\lfloor\frac{r}{2}\right\rfloor} \sum_{z=1}^{(r-2u)\wedge(r-2v)}
n^{-2H(u+v)} \sum_{i,j=0}^{\nT -1} \Big| \E\Big[ Y_i(\phi) Y_j (\phi)\\
      &&   \times \, I_{2r-2(u+v)-2z} \big(\partial_{\frac im }^{\otimes r-2u-z} \widetilde{\otimes} \partial_{\frac jn }^{\otimes r-2v-z}\big) \big\langle\partial_{\frac in},\partial_{\frac jn }\big\rangle_\mathfrak{H}^z \Big] \Big| \\
      & =:&C( D_1+D_2).
    \end{eqnarray*}

    We first study term $D_2$, that is when $z\geq1$. On one hand, from the estimate (\ref{j5}), we get
    \begin{eqnarray}  \notag
      \Big\|  I_{2r-2(u+v)-2z} \big(\partial_{\frac in}^{\otimes r-2u-z} \widetilde{\otimes} \partial_{\frac jn }^{\otimes r-2v-z}\big) \Big\|_{L^{q/(q-2)}(\Omega)} &\leq& C \big( \|\partial_{\frac in}\|_\mathfrak{H}^{r-2u-z} \,\|\partial_{\frac jn}\|_\mathfrak{H}^{r-2v-z} \big) \\ \notag
      &=& C \,\|\partial_{\frac 1n}\|_\mathfrak{H}^{2r-2(u+v)-2z} \\  \label{eq5a}
      &=& C \,n^{-2H(r-u-v-z)}.
    \end{eqnarray}
    On the other hand,  using (\ref{eq4a}), we obtain
        \begin{equation} \label{eq6a}
          \sum_{i,j=0}^{\nT-1} \Big| \big\langle\partial_{\frac in},\partial_{\frac jn}\big\rangle_\mathfrak{H}^z \Big| \leq C \,n^{1-2zH}.
          \end{equation}
    Thus, from (\ref{eq5a}) and (\ref{eq6a}) and using H\"older's inequality, we deduce that the term $D_2$ is bounded by
    \begin{eqnarray*}
      D_2 &\leq&
       C\sum_{u,v=0}^{\left\lfloor\frac{r}{2}\right\rfloor} \sum_{z=1}^{(r-2u)\wedge(r-2v)}
       \,n^{-2H(u+v)} \sup_{0\leq j\leq \nT -1} \|Y_j(\phi)\|_{L^q(\Omega)}^2 \,n^{-2H(r-u-v-z)} \,n^{1-2zH} \\
      &\leq& C \sup_{0\leq j\leq \nT -1} \|Y_j(\phi)\|_{L^q(\Omega)}^2 \,n^{1-2rH}.
    \end{eqnarray*}

    Now, let us study term $D_1$, that is when $z=0$.
    By (\ref{dual}) we have
    \begin{eqnarray*}
      && \Big| \E\Big[ Y_i(\phi) Y_j(\phi) \, I_{2(r-u-v)} \big(\partial_{\frac in}^{\otimes r-2u} \widetilde{\otimes} \partial_{\frac jn }^{\otimes r-2v}\big) \Big] \Big| \\
      &&  \qquad  =\left | \E\Big[ \big\langle D^{2(r-u-v)} (Y_i (\phi)Y_j(\phi)), \partial_{\frac in}^{\otimes r-2u} \widetilde{\otimes} \partial_{\frac jn}^{\otimes r-2v} \big\rangle_{\mathfrak{H}^{\otimes 2(r-u-v)}} \Big] \right|.
    \end{eqnarray*}
 Write $s=2(r-u-v)$. By definition of Malliavin derivative and Leibniz rule, $D^{s}_{u_1,\ldots,u_s} (Y_i(\phi) Y_j(\phi))$ consists of terms of the form
    $D_{{\bf u}_J}^{|J|}(Y_i(\phi))D_{{\bf u}_{J^c}}^{s-|J|}(Y_j(\phi))$, where $J$ is a subset of $\{1,\ldots,s\}$, $|J|$ denotes the cardinality of $J$ and ${\bf u}_J=(u_i)_{i\in J}$.
    Without loss of generality, we may fix $J$ and assume that $a=|J|\geq1$.
     By our assumptions on $\phi$ and the definition of Malliavin derivative, we know that
    \begin{eqnarray*}
     D^a(Y_i(\phi) )&=& \phi^{(a)}(\widetilde{B}^H_{\frac in })\widetilde{\varepsilon}_{\frac in}^{\otimes a} - \phi^{(a)}(B^H_{\frac {k(i)}m})\varepsilon_{\frac {k(i)}m}^{\otimes a}
      =  Y_i(\phi^{(a)})  \varepsilon_{\frac {k(i)}m}^{\otimes a} +  \phi^{(a)}(\widetilde{B}^H_{\frac in }) \big( \widetilde{\varepsilon}_{\frac in }^{\otimes a}- \varepsilon_{\frac {k(i)}m}^{\otimes a} \big),
    \end{eqnarray*}
    where recall that $k=k(i)=\sup\{ j: \frac jm \le \frac in \}$,
    and, for each $a\leq 2r$, we have $D^a(Y_i(\phi))\in L^2(\Omega;\mathfrak{H}^{\otimes a})$. Setting $b=s-|J|=s-a$ and with a slight abuse of notation, it follows that
    \begin{eqnarray*}
      && \E\Big[ \big\langle D_{{\bf u}_J}^{a}(Y_i(\phi))D_{{\bf u}_{J^c}}^{b}(Y_j(\phi)), \partial_{\frac in}^{\otimes r-2u} \otimes \partial_{\frac jn}^{\otimes r-2v} \big\rangle_{\mathfrak{H}^{\otimes 2r-2(u+v)}} \Big] \\
      &\leq&  \|Y_i(\phi^{(a)})\|_{L^2(\Omega)} \|Y_j(\phi^{(b)})\|_{L^2(\Omega)}
      \Big|\big\langle  \varepsilon_{\frac {k(i)} m}^{\otimes a} ({\bf u}_J) \otimes\varepsilon_{\frac {k(j)}m}^{\otimes b}({\bf u}_{J^c}) ,\partial_{\frac in}^{\otimes r-2u} \otimes \partial_{\frac jn }^{\otimes r-2v}  \big\rangle_\mathfrak{H^{\otimes s}}\Big| \\
      &&+   \|Y_i(\phi^{(a)})\|_{L^2(\Omega)} \|\phi^{(b)}(\widetilde{B}^H_{\frac jn })\|_{L^2(\Omega)}  \\
      &&\qquad \times
 \Big|\big\langle \varepsilon_{\frac {k(i)}m }^{\otimes a} ({\bf u}_J)\otimes \big(\widetilde{\varepsilon}_{\frac jn }^{\otimes b} -\varepsilon_{\frac {k(j)}m}^{\otimes b}\big)({\bf u}_{J^c}) ,\partial_{\frac in}^{\otimes r-2u} \otimes \partial_{\frac jn}^{\otimes r-2v}  \big\rangle_\mathfrak{H^{\otimes s}}\Big|
     \\
      &&+  \|\phi^{(a)}(\widetilde{B}^H_{\frac im})\|_{L^2(\Omega)}  \|Y_j(\phi^{(b)})\|_{L^2(\Omega)}  \\
      &&\qquad\times
    \Big|\big\langle
    \big(\widetilde{\varepsilon}_{\frac im}^{\otimes a}- \varepsilon_{\frac {k(i)}n}^{\otimes a}\big)({\bf u}_{J}) \otimes
    \varepsilon_{\frac {k(j)} n}^{\otimes b} ({\bf u}_{J^c})  ,\partial_{\frac im }^{\otimes r-2u} \otimes \partial_{\frac jm}^{\otimes r-2v}  \big\rangle_\mathfrak{H^{\otimes s}}\Big|  \\
      &&+  \|\phi^{(a)}(\widetilde{B}^H_{\frac in})\|_{L^2(\Omega)} \|\phi^{(b)}(\widetilde{B}^H_{\frac jn})\|_{L^2(\Omega)} \\
&&\qquad  \times \Big|\big\langle
    \big(\widetilde{\varepsilon}_{\frac in }^{\otimes a}-\varepsilon_{\frac {k(i)}m}^{\otimes a}\big)({\bf u}_{J}) \otimes
   \big(\widetilde{\varepsilon}_{\frac jn}^{\otimes b}-\varepsilon_{\frac {k(j)} m}^{\otimes b}\big)({\bf u}_{J^c})   ,\partial_{\frac in}^{\otimes r-2u} \otimes \partial_{\frac jn}^{\otimes r-2v}  \big\rangle_\mathfrak{H^{\otimes s}}\Big|  \\
      &=:& D_{11}+D_{12}+D_{13}+D_{14}.
    \end{eqnarray*}
   Consider first the term $D_{11}$. By (\ref{prop_fBM2}), we have either
        \[
         D_{11} \leq C \Big|\big\langle  \varepsilon_{\frac {k(i)}m} ,\partial_{\frac jn } \big\rangle_\mathfrak{H}\Big|\,\,n^{-2H(a+b-1)}\,\,\sup_{0\leq w\leq 2r}\sup_{0\leq j\leq \nT -1}  \|Y_j(\phi^{(w)})\|_{L^2(\Omega)}^2
         \]
        or
        \[
         D_{11} \leq C \Big|\big\langle  \varepsilon_{\frac {k(i)}m},\partial_{\frac in} \big\rangle_\mathfrak{H}\Big| \,\,n^{-2H(a+b-1)}\,\,\sup_{0\leq w\leq 2r}\sup_{0\leq j\leq \nT -1} \|Y_j(\phi^{(w)})\|_{L^2(\Omega)}^2.
         \]
    By Lemma \ref{lemma_fBM}.a
        \begin{equation} \label{eq7a}
         \sum_{j=0}^{\nT -1} \Big|\big\langle  \varepsilon_{\frac {k(i)}m},\partial_{\frac jn } \big\rangle_\mathfrak{H}\Big| \leq C
        \end{equation}
        and by (\ref{2.6}),
      \begin{equation} \label{eq8a}
        \sum_{i=0}^{\nT-1} \sum_{j=0}^{\nT-1} \Big|\big\langle \varepsilon_{\frac {k(i)}m},\partial_{\frac in} \big\rangle_\mathfrak{H}\Big| \Big|\big\langle \varepsilon_{\frac {k(j)}m},\partial_{\frac jn} \big\rangle_\mathfrak{H}\Big| \leq C {\color{black}{ m^{2-4H} }}.
        \end{equation}
    As a consequence, inequalities (\ref{eq7a}) and (\ref{eq8a}) imply
    \begin{eqnarray*}
      && \sum_{u,v=0}^{\left\lfloor\frac{r}{2}\right\rfloor} \,n^{-2H(u+v)} \sum_{i,j=0}^{\nT -1} D_{11} \\
      && \hspace{10mm} \leq C \,\,\sup_{0\leq w\leq 2r}\sup_{0\leq j\leq \nT-1} \|Y_j(\phi^{(w)})\|_{L^2(\Omega)}^2\,\, \sum_{u,v=0}^{\left\lfloor\frac{r}{2}\right\rfloor} \sum_{i=0}^{\nT-1} \,n^{-2H(u+v+a+b-1)} \\
      && \hspace{20mm} +\, C \,\,\sup_{0\leq w\leq 2r}\sup_{0\leq j\leq \nT-1} \|Y_j(\phi^{(w)})\|_{L^2(\Omega)}^2\,\, \sum_{u,v=0}^{\left\lfloor\frac{r}{2}\right\rfloor} \,n^{-2H(u+v+a+b-1)} \,n^{2H} \,m^{2-4H} \\
      && \hspace{10mm} \leq C \,\,\sup_{0\leq w\leq 2r}\sup_{0\leq j\leq \nT-1} \|Y_j(\phi^{(w)})\|^2_{L^2(\Omega)}\, \,\left( 1+ n^{1-2H} m^{2-4H} \right)n^{1-2rH},    \end{eqnarray*}
    where we used that $u+v+a+b-1=2r-(u+v)-1\geq r$, since $u+v+1\leq 2\left\lfloor\frac{r}{2}\right\rfloor+1=r$ for any odd integer $r$.

    We apply the same calculation to $D_{12}$ and $D_{13}$, and we similarly obtain that
    \begin{eqnarray*}
      && \sum_{u,v=0}^{\left\lfloor\frac{r}{2}\right\rfloor} \,n^{-2H(u+v)} \sum_{i,j=0}^{\nT-1} (D_{12} +D_{13}) \\
      && \hspace{10mm} \leq  C \,\,\sup_{0\leq w\leq 2r}\sup_{0\leq j\leq \nT-1} \|\phi^{(w)}(\widetilde{B}_{\frac jn })\|_{L^2(\Omega)} \sup_{0\leq j\leq \nT-1} \|Y_{j}(\phi^{(w)})\|_{L^2(\Omega)} \\
      && \hspace{25mm} \times (1 + n^{1-2H} \,m^{2-4H}) \,n^{1-2rH}.
    \end{eqnarray*}

    Now we study term $D_{14}$. Inequalities (\ref{2.7}) and (\ref{2.8}) state that
        $$ \sum_{j=0}^{\nT-1} \Big|\big\langle \widetilde{\varepsilon}_{\frac in}-\varepsilon_{\frac {k(i)}m},\partial_{\frac jn} \big\rangle_\mathfrak{H}\Big| \leq C m^{-2H}\quad\mbox{and}\quad
      \sum_{i=0}^{\nT-1} \Big|\big\langle \widetilde{\varepsilon}_{\frac in }-\varepsilon_{\frac {k(i)}m},\partial_{\frac in} \big\rangle_\mathfrak{H}\Big| \leq C m^{1-2H}.   $$
    Then, with the same reasoning used for $D_{11}$, we obtain
    \begin{eqnarray*}
       && \sum_{u,v=0}^{\left\lfloor\frac{r}{2}\right\rfloor}  \,n^{-2H(u+v)} \sum_{i,j=0}^{\nT-1} D_{14} \\
       && \hspace{10mm} \leq C \sup_{0\leq w\leq 2r} \sup_{0\leq j\leq \nT-1} \|\phi^{(w)}(\widetilde{B}^H_{\frac jn})\|_{L^2(\Omega)}^2 (m^{-2H} + \,n^{1-2H} \,m^{2-4H}) \,n^{1-2rH}.
    \end{eqnarray*}
    The proof is now concluded.
    \end{proof}


\end{document}